\documentclass[draft]{amsart}

\usepackage{amssymb,graphicx,enumerate,amsthm}
\usepackage{multirow}

\newtheorem{theorem}{Theorem}[section]

\newtheorem{corollary}[theorem]{Corollary}

\renewcommand{\S}{\mathcal{S}}
\renewcommand{\mod}[1]{\text{ (mod }#1)}

\newcommand{\RoundTrinomial}[4]{
\left(\begin{array}{c}#1 ,#2 \\ #3\end{array}; #4\right)_2
}

\newcommand{\Jacobi}[2]{\left( \frac{#1}{#2} \right)}

\numberwithin{equation}{section}

\begin{document}

\makeatletter
\def\imod#1{\allowbreak\mkern10mu({\operator@font mod}\,\,#1)}
\makeatother

\author{Alexander Berkovich}
   	\address{Department of Mathematics, University of Florida, 358 Little Hall, Gainesville FL 32611, USA}
   	\email{alexb@ufl.edu}

\author{Ali Kemal Uncu}
\address{University of Bath, Faculty of Science, Department of Computer Science, Bath, BA2\,7AY, UK}
\email{aku21@bath.ac.uk}

\thanks{Research of the second author is partly supported by EPSRC grant number EP/T015713/1 and partly by FWF grant P-34501-N}

\title[\scalebox{.9}{New infinite hierarchies of polynomial identities related to the Capparelli partition theorems}]{New infinite hierarchies of polynomial identities related to the Capparelli partition theorems}
     
\begin{abstract} We prove a new polynomial refinement of the Capparelli's identities. Using a special case of Bailey's lemma we prove many infinite families of sum-product identities that root from our finite analogues of Capparelli's identities. We also discuss the $q\mapsto 1/q$ duality transformation of the base identities and some related partition theoretic relations.
\end{abstract}

  

%11C08		Polynomials
%11P81  	Elementary theory of partitions [See also 05A17]
%11P82  	Analytic theory of partitions
%11P83  	Partitions; congruences and congruential restrictions
%11P84  	Partition identities; identities of Rogers-Ramanujan type
%%11B65  	Binomial coefficients; factorials; $q$-identities
%05A10  	Factorials, binomial coefficients, combinatorial functions [See also 11B65, 33Cxx]
%05A15  	Exact enumeration problems, generating functions [See also 33Cxx, 33Dxx]
%05A17  	Partitions of integers [See also 11P81, 11P82, 11P83]

\date{\today}
   
%\dedicatory{}
   
\maketitle

\section{Introduction}\label{Sec:Intro}

Let $a$ and $q$ be variables and define the $q$-Pochhammer symbol $(a;q)_n :=(1-a)(1-aq)\dots(1-aq^{n-1})$ for any non-negative integer $n$. For $|q|<1$, we define $(a;q)_\infty := \lim_{n\rightarrow\infty} (a;q)_n$. For $a_i$ some finite number of variables we define the shorthand notation $(a_1,a_2,\dots,a_k;q)_n :=(a_1;q)_n(a_2;q)_n\dots (a_k;q)_n$. Finally note that $1/(q;q)_n = 0$ for all negative $n$.

A \textit{partition} of a positive integer $n$ is a non-increasing sequence of natural numbers whose sum is $n$ \cite{Theory_of_Partitions}. For example, the partitions of $4$ are $(4), (3,1), (2,2), (2,1,1),$ and $ (1,1,1,1)$. The sum of all the parts of a partition $\pi$ is called the \textit{size} of $\pi$ and it is denoted with $|\pi|$. The \textit{total number of parts} of a partition $\pi$ is denoted with $\#(\pi)$. We denote the set of all the partitions by $P$ and the set of all the partitions into distinct parts by $D$. Then, it is widely known that the generating function for the number of partitions into ordinary and distinct parts are given (both as a sum representation and as a product representation) as follows \begin{equation}\label{eq:GF_P_D}
\sum_{\pi\in P} q^{|\pi|} = \sum_{n\geq 0} \frac{q^n}{(q;q)_n} = \frac{1}{(q;q)_\infty}\text{,  and  }\sum_{\pi\in D} q^{|\pi|} = \sum_{n\geq 0} \frac{q^{n(n+1)/2}}{(q;q)_n} = (-q;q)_\infty.
\end{equation}
Notice that both generating functions' $q$-series starts with the constant term 1. This constant represents the conventional partition of 0. We consider the empty sequence with 0 parts to be the only parition of 0.

Let $C_m(n)$ be the number of partitions of $n$ into distinct parts where no part is congruent to $\pm m$ modulo $6$. Define $D_m(n)$ to be the number of partitions of $n$ into parts, not equal to $m$, where the minimal difference between consecutive parts is 2. In fact, the difference between consecutive parts is greater than or equal to $4$ unless consecutive parts are $3k$ and $3k+3$ (yielding a difference of 3), or $3k-1$ and $3k+1$ (yielding a difference of 2) for some $k\in\mathbb{Z}_{>0}$. 
   
In his thesis \cite{Capparelli_thesis}, S.~Capparelli stated two (then) conjectural identities, which we will present as the following theorem. \begin{theorem}\label{thm:fulcapparelli} For any non-negative integer $n$ and $m\in\{1,2\}$,\begin{equation}\nonumber
C_m(n) = D_m(n).
\end{equation}
\end{theorem} The $m=1$ case was first proven by G.~E.~Andrews \cite{Andrews_Capparelli} shortly after its debut. Two years later Lie theoretic proofs for both cases of the conjecture were supplied by Tamba and Xie \cite{Tamba} and by Capparelli \cite{Capparelli}. The Lie theoretic proofs were followed by Alladi, Andrews, and Gordon by a refinement of these identities, where they introduced restrictions on the number of occurrences of parts belonging to certain congruence classes \cite{Refinement}. In recent years some new refinements of Capparelli's identities were discovered by Dousse and Dousse joint with Lovejoy \cite{Dousse1, Dousse2}. We invite interested readers to check these resources.

After a long wait, finally in 2018, the analytic counterparts of Theorem~\ref{thm:fulcapparelli} were independently found by Kanade--Russell \cite{Kanade_Russell} and Kur\c{s}ung\"{o}z \cite{Kagan}. 

\begin{theorem}\label{thm:Capparelli_analytic}
\begin{align*}
\sum_{m,n\geq 0} \frac{q^{2m^2+6m n+6n^2}}{(q;q)_m(q^3;q^3)_n} &= (-q^2,-q^4;q^6)_\infty (-q^3;q^3)_\infty,\\
\sum_{m,n\geq 0} \frac{q^{2m^2+6m n+6n^2+m+3n}}{(q;q)_m(q^3;q^3)_n} &+\sum_{m,n\geq 0} \frac{q^{2m^2+6m n+6n^2+3m+6n+1}}{(q;q)_m(q^3;q^3)_n} = (-q,-q^5;q^6)_\infty (-q^3;q^3)_\infty.
\end{align*}
\end{theorem}

Theorem~\ref{thm:Capparelli_analytic} relied on the original proofs \cite{Capparelli, Andrews_Capparelli} of the Capparelli identities. In \cite{BU_Finite_Capparelli}, we proved polynomial identities that directly imply Theorem~\ref{thm:Capparelli_analytic} and, hence, the Capparelli identities.

\begin{theorem}\label{thm:Fin_Capparelli_Round_tri}For any $L\in\mathbb{Z}_{\geq 0}$, we have
\begin{align} 
\label{eq:Fin_Cap1_round_tri}&\sum_{m,n\geq 0} q^{2m^2+6m n+6n^2}{3(L-2n-m)\brack m}_q {2(L-2n-m)+n\brack n}_{q^3}= \sum_{j=-\infty}^{\infty} q^{3j^2 +j} \RoundTrinomial{L}{2j}{2j}{q^3},\\
\nonumber &\sum_{m,n\geq 0} q^{2m^2+6m n+6n^2+m+3n}{3(L-2n-m)+2\brack m}_q {2(L-2n-m)+n+1\brack n}_{q^3}\\ 
\nonumber&\hspace{1cm}+ \sum_{m,n\geq 0} q^{2m^2+6m n+6n^2+3m+6n+1}{3(L-2n-m)\brack m}_q {2(L-2n-m)+n\brack n}_{q^3} \\
\nonumber&\hspace{7.2cm}=\sum_{j=-\infty}^{\infty} q^{3j^2 +2j}\RoundTrinomial{L+1}{2j+1}{2j+1}{q^3},
\end{align}
where
\begin{align*}
 %\label{Binom_def}
{m+n \brack m}_q &:= \left\lbrace \begin{array}{ll}\frac{(q;q)_{m+n}}{(q;q)_m(q;q)_{n}},&\text{for }m, n \geq 0,\\
   0,&\text{otherwise,}\end{array}\right.
   \intertext{are the $q$-binomial coefficients, and} 
%\label{eq:Round_tri_def}
\RoundTrinomial{L}{b}{a}{q} &:= \sum_{j=0}^L q^{j(j+b)} {L \brack j}_q{L-j\brack j+a}_q
\end{align*} are the $q$-trinomial coefficients defined by Andrews and Baxter \cite{Andrews_Baxter}.
\end{theorem}

Later, in \cite{BU_Elementary}, we proved different finite versions of Capparelli's theorems.

\begin{theorem}\label{thm:Fin_Capparelli_Binomial}
\begin{align}\label{eq:Fin_Cap1_Binomial}
\sum_{m,n\geq 0} \frac{q^{2m^2+6m n+6n^2}(q^3;q^3)_M}{(q;q)_m (q^3;q^3)_n(q^3;q^3)_{M-2n-m}} &= \sum_{j=-M}^M q^{3j^2+j} {2M\brack M-j}_{q^3},\\
\label{eq:Fin_Cap2_Binomial}\sum_{m,n\geq 0} \frac{q^{2m^2+6m n+6n^2+m+3n}(q^3;q^3)_M}{(q;q)_m (q^3;q^3)_n(q^3;q^3)_{M-2n-m}} &+q\sum_{m,n\geq 0} \frac{q^{2m^2+6m n+6n^2+3m+6n}(q^3;q^3)_M}{(q;q)_m (q^3;q^3)_n(q^3;q^3)_{M-2n-m}}= \sum_{j=-M-1}^M q^{3j^2+2j} {2M+1\brack M-j}_{q^3},\\
\label{eq:Fin_Sum_of_capparelli's} 
\sum_{m,n\geq 0} \frac{q^{2m^2+6mn+6n^2-2m-3n}(q^3;q^3)_M}{(q;q)_m (q^3;q^3)_n(q^3;q^3)_{M-2n-m}}&(1+q^{3M}) = \sum_{j=-M}^M q^{3j^2-2j}(1+q^{3j}) {2M\brack M-j}_{q^3}.
\end{align}
\end{theorem}

As $M\rightarrow\infty$, the identities \eqref{eq:Fin_Cap1_Binomial} and \eqref{eq:Fin_Cap2_Binomial}, with the help of Jacobi triple product identity (see \eqref{eq:JTP}), prove Theorem~\ref{thm:Capparelli_analytic}. The third identity \eqref{eq:Fin_Sum_of_capparelli's} gives a new relation that equates a double sum to the sum of the products that appear in Theorem~\ref{thm:Capparelli_analytic}.

Furthermore, in a follow up work \cite{BU_infinite_Fam}, we found a doubly bounded identity that unifies both \eqref{eq:Fin_Cap1_round_tri} and \eqref{eq:Fin_Cap1_Binomial}. 

\begin{theorem}\label{thm:Seed_identity} For $L$ and $M$ non-negative integers, we have
\begin{equation}\label{eq:Seed_identity}
\sum_{\substack{i,m\geq 0,\\ i+m\equiv 0\text{ (mod 2)}}} q^{\frac{m^2+3i^2}{2}} {L+M-i\brack L}_{q^3}{3(L-i)\brack m}_{q}{2(L-i)+\frac{i-m}{2}\brack 2(L-i)}_{q^3}=\sum_{j=-\infty}^\infty q^{3j^2+j} \S\left( \hspace{-.2cm}\begin{array}{c}L,\, M \\ 2j,\, j\end{array}\hspace{-.2cm};q^3  \right),\end{equation}
where
\begin{equation*}
\S\left( \hspace{-.2cm}\begin{array}{c}L,\, M \\a,\, b\end{array}\hspace{-.2cm};q  \right) := \sum_{n\geq0} q^{n(n+a)} {M+L-a-2n\brack M}_q {M-a+b \brack n}_q {M+a-b \brack n+a}_q
\end{equation*} a refinement of $q$-trinomial coefficients first defined by Warnaar \cite{Warnaar_T}.
\end{theorem}

As $M\rightarrow \infty$, \eqref{eq:Seed_identity} reduces to \eqref{eq:Fin_Cap1_round_tri} and  as $L\rightarrow \infty$, \eqref{eq:Seed_identity} reduces to \eqref{eq:Fin_Cap1_Binomial}.

On top of it all, this unifying identity and an analogue of Bailey's Lemma helped us identify two infinite hierarchies of polynomial identities that extends \eqref{eq:Seed_identity}. One of such infinite hierarchies is the following.

\begin{theorem}\label{thm:S_hierarchy}
Let $\nu$ be a positive integer, and let $N_k = n_k+n_{k+1}+\dots +n_{\nu}$, for $k=1,2,\dots, \nu$. Then,
\begin{align*}
\nonumber\sum_{\substack{i,m,n_1,n_2,\dots,n_\nu\geq 0,\\ i+m\equiv N_1+N_2+\dots +N_\nu\text{ (mod 2)}}} &q^{\frac{m^2+3(i^2+N_1^2+N_2^2\dots+N_\nu^2)}{2}}{L+M-i\brack L}_{q^3}{L-N_1 \brack i}_{q^3} {3n_\nu \brack m}_{q}\\\
&\hspace{2cm}\times{2n_\nu + \frac{i-m-N_1-N_2-\dots-N_\nu}{2}\brack 2n_\nu}_{q^3} \prod_{j=1}^{\nu-1}{i-\sum_{l=1}^{j} N_l+n_j \brack n_j}_{q^3}\\\nonumber
&\hspace{-1cm} = \sum_{j=-\infty}^\infty q^{3{\nu+2 \choose 2}j^2+j} \S\left( \hspace{-.2cm}\begin{array}{c}L,\, M \\ (\nu+2)j,\, (\nu+1)j\end{array}\hspace{-.2cm};q^3  \right).
\end{align*}
\end{theorem}

Tending $L\rightarrow\infty$ and $M\rightarrow\infty$ in Theorem~\ref{thm:S_hierarchy} yields the following infinite hierarchy of sum-product identities \cite[Theorem~15]{BU_infinite_Fam}. 

\begin{theorem}\label{thm:End_of_s_hierarchy}  Let $\nu$ be a positive integer, and let $N_k = n_k+n_{k+1}+\dots +n_{\nu}$, for $k=1,2,\dots, \nu$. Then,
\begin{align*}
\nonumber\sum_{\substack{i,m,n_1,n_2,\dots,n_\nu\geq 0,\\ i+m \equiv N_1+N_2+\dots + N_\nu \text{ (mod 2)}}} &\frac{q^{\frac{m^2+3(i^2+N_1^2+N_2^2+\dots + N_\nu^2)}{2}}}{(q^3;q^3)_i} {3n_\nu\brack m}_{q} {2n_\nu + \frac{i-N_1-N_2-\dots-N_\nu-m}{2}\brack 2n_\nu}_{q^3} \prod_{j=1}^{\nu-1} {i- \sum_{k=1}^j N_k+n_j \brack n_j}_{q^3}\\
 &\hspace{3cm}= \frac{(q^{6{\nu+2\choose 2}},-q^{3{\nu+2\choose 2}+1},-q^{3{\nu+2\choose 2}-1};q^{6{\nu+2\choose 2}})_\infty}{(q^{3};q^{3})_\infty}.
\end{align*}
\end{theorem}

In this paper we prove a new set of polynomial identities that imply Capparelli’s theorems.

\begin{theorem}\label{thm:New_Fin_Cap} Let $L\in\mathbb{Z}_{\geq 0}$, then
\begin{align}
\label{eq:fin_Cap1}\sum_{m,n\geq 0} \frac{q^{2m^2+  6 m n + 6n^2} (q;q)_L}{(q;q)_{L-3n-2m}(q;q)_m(q^3;q^3)_n} &= \sum_{j=-L}^L \Jacobi{j+1}{3} q^{j^2} {2L\brack L-j}_q,\\
\label{eq:fin_Cap2}\sum_{m,n\geq 0} \frac{q^{2m^2+6mn+6n^2 + m+ 3n} (q;q)_L}{(q;q)_{L-3n-2m}(q;q)_m(q^3;q^3)_n} &+ q\sum_{m,n\geq 0} \frac{q^{2m^2+6mn+6n^2 + 3m+ 6n}(q;q)_L}{(q;q)_{L-3n-2m-1}(q;q)_m(q^3;q^3)_n} = \sum_{j=-L}^L \Jacobi{j+1}{3} q^{j(j+1)} {2L\brack L-j}_q,
\end{align} where $\Jacobi{\cdot}{\cdot}$ is the Jacobi symbol.
\end{theorem}

The theorem above is analogous to Theorem~\ref{thm:Fin_Capparelli_Binomial}.

Then, by appeal to a special case of Bailey's lemma, we prove various infinite  polynomial identities which tend to infinite hierarchies of sum-product identities asymptotically. One of such sum-product hierarchies is below.

\begin{theorem}\label{thm:hierarchy_Cap1_alternate} Let $f\in\mathbb{N}$, $s\in \{0,1,2,\dots,f\}$ and $N_i := n_i+n_{i+1}+\dots+n_f$ with $i=1,2\dots, f$, where $N_{f+1}:=0$, then
\begin{align}
\label{eq:hierarchy_Cap1_alternate}
\sum_{m,n,n_1,n_2,\dots,n_f\geq 0 } &\frac{q^{2m^2+6mn+6n^2 +N_1^2 + N_2^2+\dots+N_f^2+N_{f-s+1} +\dots + N_f} (q;q)_{n_f}}{(q;q)_m (q^3;q^3)_n (q;q)_{n_f-3n-2m }(q;q)_{n_1}(q;q)_{n_2}\dots(q;q)_{n_{f-1}}(q;q)_{2n_f}}\\ \nonumber
&\hspace{3cm} = \frac{(q^{f+1-s},q^{5f+5+s},q^{6(f+1)};q^{6(f+1)})_\infty(q^{4f+4+2s},q^{8f+8-2s};q^{12(f+1)})_\infty}{(q;q)_\infty}.
\end{align}
\end{theorem}

The organization of this paper is as follows. In Section~\ref{Sec:Background}, we give useful formulas that will aid in the proofs. Section~\ref{Sec:New_Pol} has the proof of Theorem~\ref{thm:New_Fin_Cap} and other similar formulas that will lead to the discovery of infinite hierarchies. These infinite hierarchies and the proof of Theorem~\ref{thm:hierarchy_Cap1_alternate} are given in Section~\ref{Sec:hierarchy}. In Section~\ref{Sec:dual}, we look at the dual identities to the ones in Theorem~\ref{thm:New_Fin_Cap} and some related partition theoretic consequences. The last Section has some short concluding remarks.

\section{Some Useful Formulas}\label{Sec:Background}

For completeness, in this section we would also like to present some essential ingredients of our proofs. We start with two well known limits of the $q$-binomial coefficients. For any $j\in \mathbb{Z}_{\geq0}$ and $a =0$ or 1,
\begin{align}
\nonumber%\label{eq:Binom_limit}
 \displaystyle \lim_{L\rightarrow\infty}{L\brack j}_q &= \frac{1}{(q;q)_m},\\
  \nonumber%\label{eq:Binom_limit2} 
  \displaystyle \lim_{L\rightarrow\infty}{2L+a\brack L-j}_q &= \frac{1}{(q;q)_\infty}, \intertext{and for $n,\ m\in\mathbb{Z}_{\geq0}$}
\label{eq:qBinom_1_over_q}{n+m\brack m}_{q^{-1}} &= q^{-mn} {n+m\brack m}_{q}.\intertext{We would also like to recall the $q$-binomial recurrences \cite[I.45, p.353]{Gasper_Rahman}:}
 \label{eq:Binom_rec} {m+n \brack m}_q &= {m+n-1 \brack m}_q + q^{n} {m+n-1 \brack m-1}_q.\intertext{It is easy to verify that }
 \label{eq:Binom_Shift}  (1-q^j) {L\brack j}_q &= (1-q^L) {L-1\brack j-1}_q.
\end{align}
The \textit{$q$-binomial theorem} \cite[II.3, p.354]{Gasper_Rahman} states that \begin{equation}
\label{eq:qBin_thm} \sum_{n\geq 0 } \frac{(a;q)_n}{(q;q)_n} z^n = \frac{(az;q)_\infty}{(z;q)_\infty}.
\end{equation}
For $z\not = 0$, the \textit{Jacobi triple product identity} \cite[(1.6.1),p.15]{Gasper_Rahman} is 
\begin{equation}\label{eq:JTP}
\sum_{j=-\infty}^{\infty} q^{j^2} z^{j} = (-zq,-q/z,q^2;q^2)_\infty .
\end{equation}
The \textit{quintuple product identity} \cite[ex.5.6, p.147]{Gasper_Rahman} for $z\not= 0$ is given as  \begin{equation}\label{eq:quintuple}
\sum_{j=-\infty}^{\infty} (-1)^j q^{j(3j-1)/2} z^{3j}(1+zq^j) = (q,-z,-q/z;q)_\infty (qz^2,q/z^2;q^2)_\infty.
\end{equation}

We now present a special case of Bailey's Lemma \cite{GEA_multi, Peter_Thesis} that will be instrumental to some of our proofs.

\begin{theorem}\label{thm:Bailey} For $a=0,1$, if
\[F_a(L,q) = \sum_{j=-\infty}^\infty \alpha_j(q) {2L+a\brack L-j}_q\] then
\[\sum_{r\geq 0 } \frac{q^{r^2+ar} (q;q)_{2L+a}}{(q;q)_{L-r}(q;q)_{2r+a}} F_a(r,q) = \sum_{j=-\infty}^\infty \alpha_j(q) q^{j^2+aj} {2L+a\brack L-j}_q.\]
\end{theorem}

Observe that the right-hand side of the second equation in the Theorem~\ref{thm:Bailey} is of the same form as the right-hand side of the first equation. Thus, we may iterate Theorem~\ref{thm:Bailey} as often as we desire by updating $\alpha_{j}(q)$'s in each step. This procedure gives rise to an infinite hierarchy of polynomial identities.

\section{New Polynomial Identities}\label{Sec:New_Pol}

%\sum_{k=-\infty}^\infty \left(q^{(3k)^2} {2L\brack L+3k}_q - q^{(3k+1)^2} {2L\brack L+3k+1}_q  \right)\\
%\sum_{k=-\infty}^\infty \left( q^{3k(3k+1)} {2L\brack L+3k}_q - q^{(3k+1)(3k+2)} {2L\brack L+3k+1}_q  \right)

We prove Theorem~\ref{thm:New_Fin_Cap} by showing that both sides of \eqref{eq:fin_Cap1} and \eqref{eq:fin_Cap2} satisfy the recurrences \begin{align}
\label{eq:Rec_a_n_short}a_L &= (1+q-q^{2L-1})a_{L-1} -q \left(1-q^{2L-2}\right) \left(1-q^{2L-3}\right) a_{L-2},\intertext{and}
\label{eq:Rec_b_n_short}b_L &= (1+q-q^{2L})b_{L-1} -q\left(1-q^{2L-1}\right)\left(1-q^{2L-2}\right)  b_{L-2},
\end{align}respectively, and by checking the appropriate  initial conditions.

\begin{proof}

The $q$-Zeilberger algorithm \cite{AeqB} implemented in Riese's \texttt{qZeil} package \cite{qZeil} (for single-fold sums), Sister Celine's algorithm implemented in Riese's \texttt{qMultiSum} package \cite{qMultiSum}, Creative Telescoping implemented in the Mathematica package \texttt{HolonomicFunctions} of Koutschan \cite{HolonomicFunctions}, and Schneider's \texttt{Sigma} package \cite{Sigma}) are all sufficient tools to find and prove recurrences satisfied by the expressions in \eqref{eq:fin_Cap1} and \eqref{eq:fin_Cap2}. However, the outcome recurrences heavily depend on the representation of these functions. For that reason, we first rewrite the right-hand side expressions first
\begin{align}
\label{eq:finCap1_right_side_1}\sum_{j=-L}^L \Jacobi{j+1}{3} q^{j^2} {2L\brack L+j}_q &=\sum_{k=-\infty}^\infty \left(q^{(3k)^2} {2L\brack L+3k}_q - q^{(3k+1)^2} {2L\brack L+3k+1}_q  \right),\\
\label{eq:finCap1_right_side_2}&=\sum_{j=-\lfloor\frac{L}{3}\rfloor }^{\lfloor\frac{L}{3}\rfloor} q^{(3j)^2} \frac{1-q^{6j+1}}{1-q^{L+3j+1}} {2L\brack L+3j}_q - \delta_{3 \mid (L-2)} q^{L^2},\\
\label{eq:finCap2_right_side_1}\sum_{j=-L}^L \Jacobi{j+1}{3} q^{j(j+1)} {2L\brack L+j}_q &=\sum_{k=-\infty}^\infty \left( q^{3k(3k+1)} {2L\brack L+3k}_q - q^{(3k+1)(3k+2)} {2L\brack L+3k+1}_q  \right),\\
\label{eq:finCap2_right_side_2}&=\sum_{j=-\lfloor\frac{L}{3}\rfloor }^{\lfloor\frac{L}{3}\rfloor} q^{(3j)(3j-1)} \frac{1-q^{6j+2}-(1-q)q^{L+3j+1}}{1-q^{L+3j+1}} {2L\brack L+3j}_q - \delta_{3 \mid (L-2)} q^{L(L-1)},
\end{align} where $\delta_{a \mid b} = 1$ if $a\mid b$, and $0$ otherwise.

If one chooses to calculate the recurrence from the representations \eqref{eq:finCap1_right_side_1} and \eqref{eq:finCap2_right_side_1}, the orders of the recurrences found/proven would be 4 and 6, respectively. 

In certain cases, we find and prove better recurrences.
When we start with the alternative representations \eqref{eq:finCap1_right_side_2} we get the recurrence \eqref{eq:Rec_a_n_short}.
%\begin{align*}
%a_L &= (1+q-q^{2L-1})a_{L-1} -q \left(1-q^{2L-2}\right) \left(1-q^{2L-3}\right) a_{L-2}.
%\end{align*}

With the alternative representation \eqref{eq:finCap2_right_side_2} of the right-hand side of \eqref{eq:fin_Cap2}, which is analogous to \eqref{eq:finCap1_right_side_1}, we can only find a 4th order recurrence \eqref{eq:Rec_b_n_proven} using the automated proof techniques:
\begin{align}\label{eq:Rec_b_n_proven}
b_L &=  -(1 + q^2) (1 + q - q^{2 L-2}) b_{L-1}+q ((1 + q^2)(1+q+q^2)-2q^{2L-3}(1+q)(1+q^2)+q^{4L-7}(1+q^2+q^4)) b_{L-2}\\
    \nonumber&-
 q^3 (1 + q^2) (1 - q^{2L-4})  (1 - q^{2 L-5}) (1 + q - q^{2L-4}) b_{L-3} 
+q^6 (1 - q^{2L-4})(1 - q^{2 L-5})(1 - q^{2L-6})(1 - q^{2 L-7})  b_{L-4}.
\end{align} 

However, we can do better than that. Experimentally, using the \texttt{qFunctions} package \cite{qFunctions} of the second author and Ablinger, we can guess that the right-hand side of \eqref{eq:fin_Cap2} satisfies the 2nd order recurrence \eqref{eq:Rec_b_n_short}.
%\begin{align*}
%b_L &= (1+q-q^{2L})b_{L-1} -q\left(1-q^{2L-1}\right)\left(1-q^{2L-2}\right)  b_{L-2}.
%\end{align*}
Moreover, again by using \texttt{qFunctions} package \cite{qFunctions} we can prove that the greatest common divisor of \eqref{eq:Rec_b_n_proven} and \eqref{eq:Rec_b_n_short} is \eqref{eq:Rec_b_n_short}. To demonstrate that \eqref{eq:Rec_b_n_short} is a factor of \eqref{eq:Rec_b_n_proven}, let \begin{equation}\label{eq:Rec_r_L}r_L:= b_L - (1+q-q^{2L})b_{L-1} +q\left(1-q^{2L-1}\right)\left(1-q^{2L-2}\right)  b_{L-2}.\end{equation} Then it is easy to observe that \[r_L - q^2(1 + q - q^{2 L - 4}) r_{L-1} + q^5(1 - q^{2L-4}) (1 - q^{2 L-7}) r_{L-2}=0\] is equivalent to \eqref{eq:Rec_b_n_proven}. 

By checking the initial conditions, this is enough to prove that the right-hand sides of \eqref{eq:fin_Cap2} satisfy the guessed recurrence \eqref{eq:Rec_b_n_short}. This is done for $b_n$ as follows: We know that \eqref{eq:Rec_b_n_proven} and $4$ initial values fully determine the sequence $b_n$. We also know that \eqref{eq:Rec_b_n_short} divide this sequence and only require $2$ initial conditions to determine a sequence fully. We feed the first two initial conditions of \eqref{eq:Rec_b_n_proven} to \eqref{eq:Rec_b_n_short} and check that the next 2 sequence terms we get from \eqref{eq:Rec_b_n_short} are the same as the next 2 terms for \eqref{eq:Rec_b_n_short}. This way we prove that not only that the guessed recurrence \eqref{eq:Rec_b_n_short} divides the proven recurrence \eqref{eq:Rec_b_n_proven} that \eqref{eq:fin_Cap2} satisfies, but the recurrence \eqref{eq:Rec_b_n_short} also determines the same sequence fully with the 2 initial conditions of \eqref{eq:Rec_b_n_proven}.

Now, using Zeilberger's Creative Telescoping algorithm \cite{AeqB} (implemented in Koutschan's \texttt{HolonomicFunctions} package \cite{HolonomicFunctions}) we find the recurrences satisfied by the left-hand sides of \eqref{eq:fin_Cap1} and \eqref{eq:fin_Cap2}. The said implementation directly proves that, just like the right-hand side, the left-hand side of \eqref{eq:fin_Cap1} satisfies \eqref{eq:Rec_a_n_short}. Therefore, since both sides of \eqref{eq:fin_Cap1} satisfy the same 2nd order recurrence, we finish the proof of that identity by checking the two initial conditions at $L=0$ and 1.

For finding the recurrence of the left-hand side of \eqref{eq:fin_Cap2}, first we define\[S_{1,L} :=  \sum_{m,n\geq 0} \frac{q^{2m^2+  6 m n + 6n^2+m+3n} (q;q)_L}{(q;q)_{L-3n-2m}(q;q)_m(q^3;q^3)_n}\text{ and }S_{2,L} :=\sum_{m,n\geq 0 }  \frac{q^{2m^2+  6 m n + 6n^2+3m+6n+1} (q;q)_L}{(q;q)_{L-3n-2m-1}(q;q)_m(q^3;q^3)_n}. \] Then, we get that these sums satisfy the recurrences
\begin{align}
\label{eq:rec_S1}S_{1,L} &=  (1 + q + q^3 - q^L - q^{L+2})S_{1,L-1}\\
\nonumber&-q(1 - q^{L-1}) (1 + q^2 + q^3 - q^{L+1} - q^{2 L-1}- q^{2 L-2}) S_{1,L-2}\\
\nonumber&+q^4(1 - q^{L-1}) (1 - q^{L-2}) (1- q^{2 L-3})(1 - q^{2L-4}) S_{1,L-3}
\intertext{and}
\label{eq:rec_S2}(1-q^{L-1})S_{2,L} &= -(1 - q^L)(1+q+q^2-q^{L-1}-q^L)S_{2,L-1}\\ \nonumber&+  q(1 - q^L)(1 + q^{L-1}) (1 + q + q^2  - q^{ L-1} - q^{2 L-1}- q^{2 L-2}) S_{2,L-2}\\
    \nonumber
&-q^3(1 - q^L) (1 - q^{L-1}) (1- q^{L-2})  (1 - q^{2 L-3}) (1 - q^{2L-4})S_{2,L-3}.
\end{align}

These recurrences are similar but unlike the situation we encountered on the right-hand sides of \eqref{eq:finCap1_right_side_1} and \eqref{eq:finCap2_right_side_1}, they are not identical. 

We can find a recurrence that is satisfied by both sums on the left-hand side of \eqref{eq:fin_Cap2} by using the closure properties of holonomic (recurrent) sequences. This is merely a specialized substitution of one of the recurrences \eqref{eq:rec_S1} and \eqref{eq:rec_S2} into the other one. This is implemented in the \texttt{qGeneratingFunctions} package of Kauers and Koutschan \cite{qGeneratingFunctions}. This way we prove that the whole left-hand side satisfies the 5th order recurrence 
\begin{align}\label{eq:Rec_c_L}
c_L &=(1 + q + q^2 + q^3 + q^4 - q^{2 L-1} - q^L - q^{L+2}) c_{L-1}\\ \nonumber
&-q \left((1 + q^2) (1 + q + q^2 + q^3 + q^4)- q^{L-1} (1 + q) (1 + q^2)^2\right.\\\nonumber&\hspace{4cm}\left.- q^{2 L-2} (1 + q) (2 + q^2)+ q^{3 L-3} (1 + q)^2  +q^{4 L-5}\right) c_{L-2}\\
\nonumber&+  q^3 (1 - q^{L-2})\left((1 + q^2) (1 + q + q^2 + q^3 + q^4)- q^L (1 + q^2)  - 
 q^{2 L-4} (1 + q) (1 + 2 q + 2 q^2 + q^3 + q^4)\right.\\\nonumber&\hspace{4cm}\left.- q^{3 L-5} (1 + q) (1 - 2 q) + q^{4 L-7} (2 + 2 q + q^2)-q^{5 L-7}\right)c_{L-3}  \\
\nonumber&-q^6 (1 - q^{L-2}) (1 - q^{2L-6}) \left(1 - q^{2 L-5}) 
  ((1 + q + q^2 + q^3 + q^4)\right.\\\nonumber&\hspace{4cm}\left. - q^{L-3} (1 + q^2 + q^3) - q^{2 L-5} (1 + q + q^2) +
  q^{3 L-6}\right)c_{L-4}\\\nonumber
&+q^{10} (1 - q^{L-2})(1 - q^{L-4}) (1 - q^{2 L-5}) (1 - q^{2L-6}) (1 - q^{2 L-7}) (1 - q^{2L-8}) c_{L-5}. 
\end{align}

Comparing this with \eqref{eq:Rec_b_n_short}, we see that the greatest common divisor of these two recurrences is \eqref{eq:Rec_b_n_short}. One can easily verify this by checking that \begin{align*}r_L&-q^2(1 + q + q^2 - q^{L-2} + q^{L} + q^{2L -2} + q^{2 L-3}) r_{L-1}\\& + q^5(1-q^{L-2})(1+q+q^2-q^{L-3}-q^{2L-4}-q^{2L-5}+q^{3L-6}-q^{3L-7})r_{L-2}\\&+q^9(1 - q^{L-2}) (1 - q^{L-4}) (1 - q^{2 L-5})(1 - q^{2 L-6})r_{L-3} = 0
\end{align*} is equivalent to \eqref{eq:Rec_c_L}, where $r_L$ is defined as in \eqref{eq:Rec_r_L}.

Analogous to moving from the 4th degree recurrence \eqref{eq:Rec_b_n_proven} to the shorter recurrence \eqref{eq:Rec_b_n_short}, we prove that the left-hand side of \eqref{eq:fin_Cap2} satisfies this much shorted recurrence. This proves that the both the left- and right-hand sides of \eqref{eq:fin_Cap2} satisfies \eqref{eq:Rec_b_n_short} and by only checking the two initial conditions $L=0$ and $1$ we finish the proof of \eqref{eq:fin_Cap2}.
\end{proof}

Furthermore, we can prove a simple transformation formula for the right-hand side sum of \eqref{eq:fin_Cap2}.

\begin{theorem}
Let $L\in\mathbb{Z}_{\geq 0}$, then
\begin{align}
\label{eq:fin_Cap2_RHS_qBin_alternative}\sum_{j=-L}^L \Jacobi{j+1}{3} q^{j(j+1)} {2L\brack L-j}_q &= \sum_{j=-L-1}^{L+1} \Jacobi{j+1}{3} q^{j(j+1)} {2L+1\brack L-j}_q.
\end{align}
\end{theorem}

\begin{proof} We start by applying \eqref{eq:Binom_rec} to the right-hand side of \eqref{eq:fin_Cap2_RHS_qBin_alternative}:
\[
\sum_{j=-L-1}^{L+1} \Jacobi{j+1}{3} q^{j(j+1)} {2L+1\brack L-j}_q = \sum_{j=-L}^L \Jacobi{j+1}{3} q^{j(j+1)} {2L\brack L-j}_q + \sum_{j=-L-1}^{L+1} \Jacobi{j+1}{3} q^{j(j+1) + (L+j+1)} {2L\brack L-j-1}_q.
\]
Therefore, it is enough to prove that the rightmost series above vanishes. With the change of variables $j\mapsto j-1$, we get
\[
\sum_{j=-L-1}^{L+1} \Jacobi{j+1}{3} q^{j(j+1) + (L+j+1)} {2L\brack L-j-1}_q = q^L\sum_{j=-L}^{L} \Jacobi{j}{3} q^{j^2} {2L\brack L-j}_q.
\]
We point out that \[\Jacobi{j}{3} = -\Jacobi{j}{3},\text{  and  }{2L\brack L-j}_q = {2L\brack L+j}_q.\] Hence, by changing $j\mapsto -j$, we can now clearly see that \[\sum_{j=-L}^{L} \Jacobi{j}{3} q^{j^2} {2L\brack L-j}_q = - \sum_{j=-L}^{L} \Jacobi{j}{3} q^{j^2} {2L\brack L-j}_q = 0.\]
\end{proof}

Now, we move onto the proof of a series transformation involving $q$-Binomial coefficients.

\begin{theorem}\label{thm:k_transform} Let $L\in\mathbb{Z}_{\geq 0}$ and $k=1,2,\dots$, then
\begin{equation}
\label{eq:k_transform} \sum_{j=-L}^L \Jacobi{j+1}{3} q^{kj(j-1)} {2L\brack L-j}_q = q^L\sum_{j=-L}^L \Jacobi{j+1}{3} q^{kj^2-(k-1)j} {2L\brack L-j}_q.
\end{equation}
\end{theorem}

\begin{proof}
We take the difference of the two sides of \eqref{eq:k_transform} and show that it reduces to zero. Making use of \eqref{eq:Binom_Shift}, we have
\begin{align*}
\sum_{j=-L}^L \Jacobi{j+1}{3} q^{kj(j-1)} (1-q^{L+j}) {2L\brack L+j}_q &= (1-q^{2L}) \sum_{j=-L}^L \Jacobi{j+1}{3} q^{kj(j-1)}{2L-1\brack L+j-1}_q\\
&=(1-q^{2L}) \sum_{j=-L}^L \Jacobi{j+1}{3} q^{kj(j-1)}{2L-1\brack L-j}_q\\
&=(1-q^{2L}) \sum_{j=-L}^L \left\{ q^{3kj(3j-1)} {2L-1\brack L-3j} - q^{3kj(3j+1)} {2L-1\brack L-3j-1}_q \right\}
\end{align*}
Replacing $j\mapsto -j$ in the last term in the braces in the above term, we can notice that the terms in the braces vanish. Therefore the difference of the sides of \eqref{eq:k_transform} vanishes for each non-negative $L$ and positive integers $k$.
\end{proof}

Multiplying \eqref{eq:fin_Cap1} by $q^L$ and using \eqref{eq:k_transform} with $k=1$ on the right-hand side of the resulting equation we arrive at Corollary~\ref{cor:fin_Cap2_analogue}.

\begin{corollary}\label{cor:fin_Cap2_analogue}
Let $L\in\mathbb{Z}_{\geq 0}$, then
\begin{equation}
\label{eq:fin_Cap1_qBin_alternative}\sum_{m,n\geq 0} \frac{q^{L+2m^2+  6 m n + 6n^2} (q;q)_L}{(q;q)_{L-3n-2m}(q;q)_m(q^3;q^3)_n} = \sum_{j=-L}^L \Jacobi{j+1}{3} q^{j(j-1)} {2L\brack L-j}_q.
\end{equation}
\end{corollary}

\section{New Infinite Hierarchies}\label{Sec:hierarchy}

%\subsection{New Hierarchies Which We could have Written in Peter's volume}

Applying Teorem~\ref{thm:Bailey} with $a=0$ and $q\mapsto q^3$, $f$ times in iterative fashion to \eqref{eq:Fin_Cap1_Binomial}, we derive:

\begin{theorem}\label{thm:hierarchy_Fin_Cap1_Binomial}
Let $L\in\mathbb{Z}_{\geq0}$, $f\in\mathbb{N}$ and $N_i := n_i+n_{i+1}+\dots+n_f$ with $i=1,2\dots, f$, then
\begin{align}
\label{eq:hierarchy_Fin_Cap1_Binomial}
\sum_{m,n,n_1,n_2,\dots,n_f\geq 0 } &\frac{q^{2m^2+6mn+6n^2+3(N_1^2 + N_2^2+\dots+N_f^2)}(q^3;q^3)_{2L} (q^3;q^3)_{n_f}}{(q;q)_m (q^3;q^3)_n (q^3;q^3)_{L-N_1}(q^3;q^3)_{n_f-2n-m }(q^3;q^3)_{n_1}(q^3;q^3)_{n_2}\dots(q^3;q^3)_{n_{f-1}}(q^3;q^3)_{2n_f}}\\ \nonumber
&\hspace{3cm}  =  \sum_{j=-L}^{L} q^{3(f+1)j^2+j} {2L\brack L-j}_{q^3}.
\end{align}
\end{theorem}

And in the limit $L\rightarrow\infty$, with the help of \eqref{eq:JTP}, \eqref{eq:hierarchy_Fin_Cap1_Binomial} yields

\begin{theorem}\label{thm:hierarchy_Fin_Cap1_Binomial_Limit} Let $f\in\mathbb{N}$ and $N_i := n_i+n_{i+1}+\dots+n_f$ with $i=1,2\dots, f$, then
\begin{align}
\nonumber%\label{eq:hierarchy_Fin_Cap1_Binomial_Limit}
\sum_{m,n,n_1,n_2,\dots,n_f\geq 0 } &\frac{q^{2m^2+6mn+6n^2+3(N_1^2 + N_2^2+\dots+N_f^2)} (q^3;q^3)_{n_f}}{(q;q)_m (q^3;q^3)_n (q^3;q^3)_{n_f-2n-m }(q^3;q^3)_{n_1}(q^3;q^3)_{n_2}\dots(q^3;q^3)_{n_{f-1}}(q^3;q^3)_{2n_f}}\\ \nonumber
&\hspace{3cm}  =  \frac{(q^{6f+6}, -q^{3f+2}, -q^{3f+4}; q^{6f+6})_\infty}{(q^3;q^3)_\infty}.
\end{align}
\end{theorem}

It is clear that the products in Theorems~\ref{thm:hierarchy_Fin_Cap1_Binomial_Limit} and \ref{thm:End_of_s_hierarchy} are identical when $2f+2 = (\nu+1)(\nu+2)$. This observation implies the following transformation.

\begin{corollary} For $f = \nu(\nu+3)/2$, where $\nu$ is a positive integer, $N’_j :=n_j+ ….+ n_\nu$  for $j=1,2,\dots,\nu$ and $N_i := n_i+n_{i+1}+\dots+n_f$ with $i=1,2\dots, f$, we have
\begin{align*}
\sum_{\substack{i,m,n_1,n_2,\dots,n_\nu\geq 0,\\ i+m \equiv N'_1+N'_2+\dots + N'_\nu \text{ (mod 2)}}}& \frac{q^{\frac{m^2+3(i^2+N'_1{}^2+N'_2{}^2+\dots + N'_\nu{}^2)}{2}}}{(q^3;q^3)_i} {3n_\nu\brack m}_{q} {2n_\nu + \frac{i-N'_1-N'_2-\dots-N'_\nu-m}{2}\brack 2n_\nu}_{q^3} \prod_{j=1}^{\nu-1} {i- \sum_{k=1}^j N'_k+n_j \brack n_j}_{q^3}\\
=\sum_{m,n,n_1,n_2,\dots,n_f\geq 0 } &\frac{q^{2m^2+6mn+6n^2+3(N_1^2 + N_2^2+\dots+N_f^2)} (q^3;q^3)_{n_f}}{(q;q)_m (q^3;q^3)_n (q^3;q^3)_{n_f-2n-m }(q^3;q^3)_{n_1}(q^3;q^3)_{n_2}\dots(q^3;q^3)_{n_{f-1}}(q^3;q^3)_{2n_f}}.
\end{align*}
\end{corollary}

Applying Theorem~\ref{thm:Bailey} with $a=1$ and $q\mapsto q^3$, $f$ times in iterative fashion to \eqref{eq:Fin_Cap2_Binomial}, we derive:

\begin{theorem}\label{thm:hierarchy_Fin_Cap2_Binomial}
Let $L\in\mathbb{Z}_{\geq0}$, $f\in\mathbb{N}$ and $N_i := n_i+n_{i+1}+\dots+n_f$ with $i=1,2\dots, f$, then
\begin{align}
\label{eq:hierarchy_Fin_Cap2_Binomial}
\sum_{m,n,n_1,n_2,\dots,n_f\geq 0 } &\frac{q^{2m^2+6mn+6n^2+m+3n +3(N_1^2 + N_2^2+\dots+N_f^2+N_1+N_2+\dots+N_f)} ( 1+q^{1+2m+3n}) (q^3;q^3)_{2L+1} (q^3;q^3)_{n_f}}{(q;q)_m (q^3;q^3)_n (q^3;q^3)_{L-N_1}(q^3;q^3)_{n_f-2n-m }(q^3;q^3)_{n_1}(q^3;q^3)_{n_2}\dots(q^3;q^3)_{n_{f-1}}(q^3;q^3)_{2n_f+1}}\\ \nonumber
&\hspace{3cm}  =  \sum_{j=-L-1}^{L+1} q^{3(f+1)j^2+(3f+2)j} {2L+1\brack L-j}_{q^3}.
\end{align}
\end{theorem}

Notice that as $L\rightarrow\infty$ , with the help of \eqref{eq:JTP}, \eqref{eq:hierarchy_Fin_Cap2_Binomial} yields

\begin{theorem}Let $f\in\mathbb{N}$ and $N_i := n_i+n_{i+1}+\dots+n_f$ with $i=1,2\dots, f$, then
\begin{align}
\nonumber%\label{eq:hierarchy_Fin_Cap2_Binomial_Limit}
\sum_{m,n,n_1,n_2,\dots,n_f\geq 0 } &\frac{q^{2m^2+6mn+6n^2+m+3n +3(N_1^2 + N_2^2+\dots+N_f^2+N_1+N_2+\dots+N_f)} ( 1+q^{1+2m+3n})(q^3;q^3)_{n_f}}{(q;q)_m (q^3;q^3)_n (q^3;q^3)_{n_f-2n-m }(q^3;q^3)_{n_1}(q^3;q^3)_{n_2}\dots(q^3;q^3)_{n_{f-1}}(q^3;q^3)_{2n_f+1}}\\ \nonumber
&\hspace{3cm}  =  \frac{1}{(q^3;q^3)_\infty}\sum_{j=-\infty}^{\infty} q^{3(f+1)j^2+(3f+2)j}\\ \nonumber
&\hspace{3cm}  =  \frac{(q^{6f+6}, -q, -q^{6f+5}; q^{6f+6})_\infty}{(q^3;q^3)_\infty}.
\end{align}
\end{theorem}

We can apply Theorem~\ref{thm:Bailey} with $a=0$ and $q\mapsto q^3$ iteratively to derive the infinite family that roots from \eqref{eq:Fin_Sum_of_capparelli's}.

\begin{theorem}\label{thm:hierarchy_Fin_Sum_of_Capparellis}Let $L\in\mathbb{Z}_{\geq0}$, $f\in\mathbb{N}$ and $N_i := n_i+n_{i+1}+\dots+n_f$ with $i=1,2\dots, f$, then
\begin{align}
\label{eq:hierarchy_Fin_Sum_of_Capparellis}
\sum_{m,n,n_1,n_2,\dots,n_f\geq 0 } &\frac{q^{2m^2+6mn+6n^2-2m-3n +3(N_1^2 + N_2^2+\dots+N_f^2)} (q^3;q^3)_{2L} (q^3;q^3)_{n_f} (1+q^{3n_f})}{(q;q)_m (q^3;q^3)_n (q^3;q^3)_{L-N_1}(q^3;q^3)_{n_f-2n-m }(q^3;q^3)_{n_1}(q^3;q^3)_{n_2}\dots(q^3;q^3)_{n_{f-1}}(q^3;q^3)_{2n_f}}\\ \nonumber
&\hspace{3cm}  =  \sum_{j=-L-1}^{L+1} q^{3(f+1)j^2-2j} (1+q^{3j}) {2L\brack L-j}_{q^3}.
\end{align}
\end{theorem}

We can apply the Jacobi triple product identity \eqref{eq:JTP} to  \eqref{eq:hierarchy_Fin_Sum_of_Capparellis} twice after tending $L\rightarrow\infty$, and this yields:

\begin{theorem}\label{thm:hierarchy_Sum_of_Capparellis} Let $f\in\mathbb{N}$ and $N_i := n_i+n_{i+1}+\dots+n_f$ with $i=1,2\dots, f$, then
\begin{align}
\nonumber\sum_{m,n,n_1,n_2,\dots,n_f\geq 0 } &\frac{q^{2m^2+6mn+6n^2-2m-3n +3(N_1^2 + N_2^2+\dots+N_f^2)} (q^3;q^3)_{n_f}(1+q^{3n_f})}{(q;q)_m (q^3;q^3)_n (q^3;q^3)_{n_f-2n-m }(q^3;q^3)_{n_1}(q^3;q^3)_{n_2}\dots(q^3;q^3)_{n_{f-1}}(q^3;q^3)_{2n_f}}\\ \nonumber%\label{eq:hierarchy_Sum_of_Capparellis}
&\hspace{2cm}  = \frac{(q^{6(f+1)};q^{6(f+1)})_\infty}{(q^3;q^3)_\infty} \left( (-q^{3f+1}, -q^{3f+5};q^{6(f+1)})_\infty + (-q^{3f+2}, -q^{3f+4};q^{6(f+1)})_\infty \right) 
\end{align}
\end{theorem}

%\subsection{New Hierarchies Where everything is from this paper}
Applying Theorem~\ref{thm:Bailey} with $a=0$, $f$ times in iterative fashion to \eqref{eq:fin_Cap1} we derive:

\begin{theorem}\label{thm:hierarchy_Fin_Cap1} Let $L\in\mathbb{Z}_{\geq0}$, $f\in\mathbb{N}$ and $N_i := n_i+n_{i+1}+\dots+n_f$ with $i=1,2\dots, f$, then
\begin{align}
\label{eq:hierarchy_Fin_Cap1}
\sum_{m,n,n_1,n_2,\dots,n_f\geq 0 } &\frac{q^{2m^2+6mn+6n^2 +N_1^2 + N_2^2+\dots+N_f^2} (q;q)_{2L} (q;q)_{n_f}}{(q;q)_m (q^3;q^3)_n (q;q)_{L-N_1}(q;q)_{n_f-3n-2m }(q;q)_{n_1}(q;q)_{n_2}\dots(q;q)_{n_{f-1}}(q;q)_{2n_f}}\\ \nonumber
&\hspace{3cm}  =  \sum_{j=-L}^L \Jacobi{j+1}{3} q^{(f+1)j^2} {2L\brack L-j}_q.
\end{align}
\end{theorem}

As $L\rightarrow\infty$, with the aid of quintuple product identity \eqref{eq:quintuple}, we get 

\begin{theorem}\label{thm:hierarchy_Cap1} Let $f\in\mathbb{N}$ and $N_i := n_i+n_{i+1}+\dots+n_f$ with $i=1,2\dots, f$, then
\begin{align}
\label{eq:hierarchy_Cap1}
\sum_{m,n,n_1,n_2,\dots,n_f\geq 0 } &\frac{q^{2m^2+6mn+6n^2 +N_1^2 + N_2^2+\dots+N_f^2} (q;q)_{n_f}}{(q;q)_m (q^3;q^3)_n (q;q)_{n_f-3n-2m }(q;q)_{n_1}(q;q)_{n_2}\dots(q;q)_{n_{f-1}}(q;q)_{2n_f}}\\ \nonumber
&\hspace{3cm}  =  \frac{1}{(q;q)_\infty}\sum_{j=-\infty}^\infty \Jacobi{j+1}{3} q^{(f+1)j^2}\\ \nonumber
&\hspace{3cm} = \frac{(q^{f+1};q^{f+1})_\infty}{(q;q)_\infty} (-q^{3(f+1)};q^{3(f+1)})_\infty (-q^{2(f+1)},-q^{4(f+1)};q^{6(f+1)})_\infty.
\end{align}
\end{theorem}

Note that the summands of \eqref{eq:hierarchy_Fin_Cap1} is not necessarily made out of terms with non-negative $q$-series coefficients. In fact, some $q$-series coefficients can be negative depending on the choice of $L$, due to the $(q;q)_{2L}$ term in the numerator in the summands. As $L\rightarrow\infty$, these sign changes disappear. The summands of \eqref{eq:hierarchy_Cap1} are all manifestly positive.

Now we move onto the new infinite hierarchy related to the \eqref{eq:fin_Cap2}.  Applying Theorem~\ref{thm:Bailey} with $a=0$, $f$ times in iterative fashion to \eqref{eq:fin_Cap2} we derive:

\begin{theorem}\label{thm:hierarchy_Fin_Cap2} Let $L\in\mathbb{Z}_{\geq0}$, $f\in\mathbb{N}$ and $N_i := n_i+n_{i+1}+\dots+n_f$ with $i=1,2\dots, f$, then
\begin{align}
\nonumber
\sum_{m,n,n_1,n_2,\dots,n_f\geq 0 } &\frac{q^{2m^2+6mn+6n^2+m+3n +N_1^2 + N_2^2+\dots+N_f^2} (q;q)_{2L} (q;q)_{n_f}}{(q;q)_m (q^3;q^3)_n (q;q)_{L-N_1}(q;q)_{n_f-3n-2m }(q;q)_{n_1}(q;q)_{n_2}\dots(q;q)_{n_{f-1}}(q;q)_{2n_f}}\\ \label{eq:hierarchy_Fin_Cap2}
&+q\sum_{m,n,n_1,n_2,\dots,n_f\geq 0 } \frac{q^{2m^2+6mn+6n^2+3m+6n +N_1^2 + N_2^2+\dots+N_f^2 } (q;q)_{2L} (q;q)_{n_f}}{(q;q)_m (q^3;q^3)_n (q;q)_{L-N_1}(q;q)_{n_f-3n-2m-1}(q;q)_{n_1}(q;q)_{n_2}\dots(q;q)_{n_{f-1}}(q;q)_{2n_f}}\\ \nonumber
&\hspace{-1cm}  =  \sum_{j=-L}^L \Jacobi{j+1}{3} q^{(f+1)j^2+j} {2L\brack L-j}_q.
\end{align}
\end{theorem}

As $L\rightarrow \infty$, with the help of quintuple product identity \eqref{eq:quintuple}, we get 

\begin{theorem}\label{thm:hierarchy_Cap2} Let $L\in\mathbb{Z}_{\geq0}$, $f\in\mathbb{N}$ and $N_i := n_i+n_{i+1}+\dots+n_f$ with $i=1,2\dots, f$, then
\begin{align}
\nonumber
\sum_{m,n,n_1,n_2,\dots,n_f\geq 0 } &\frac{q^{2m^2+6mn+6n^2+m+3n +N_1^2 + N_2^2+\dots+N_f^2}  (q;q)_{n_f}}{(q;q)_m (q^3;q^3)_n (q;q)_{n_f-3n-2m }(q;q)_{n_1}(q;q)_{n_2}\dots(q;q)_{n_{f-1}}(q;q)_{2n_f}}\\ \nonumber%\label{eq:hierarchy_Cap2}
&+q\sum_{m,n,n_1,n_2,\dots,n_f\geq 0 } \frac{q^{2m^2+6mn+6n^2+3m+6n +N_1^2 + N_2^2+\dots+N_f^2 }  (q;q)_{n_f}}{(q;q)_m (q^3;q^3)_n (q;q)_{n_f-3n-2m-1}(q;q)_{n_1}(q;q)_{n_2}\dots(q;q)_{n_{f-1}}(q;q)_{2n_f}}\\ \nonumber
&\hspace{-1.5cm}  =   \frac{(q^{f+2},q^{5f+4},q^{6f+6};q^{6f+6})_\infty(q^{4f+2},q^{8f+10};q^{12f+12})_\infty}{(q;q)_\infty} .
\end{align}
\end{theorem}

In particular, for $f=1$ we have

\begin{corollary}
\begin{align}
\nonumber
\sum_{m,n,n_1\geq 0 } \frac{q^{2m^2+6mn+6n^2+m+3n +n_1^2}  (q;q)_{n_1}}{(q;q)_m (q^3;q^3)_n (q;q)_{n_1-3n-2m }(q;q)_{2n_1}}&+q\sum_{m,n,n_1\geq 0 } \frac{q^{2m^2+6mn+6n^2+3m+6n +n_1^2 } (q;q)_{n_1}}{(q;q)_m (q^3;q^3)_n (q;q)_{n_1-3n-2m-1}(q;q)_{2n_1}}\\ \nonumber%\label{eq:corollary_hierarchy_Cap2_f_eq_1}
&\hspace{7cm}  =   \frac{(q^3;q^3)_\infty}{(q;q)_\infty}.
\end{align}
\end{corollary}

Alternatively, instead of applying Theorem~\ref{thm:Bailey} to \eqref{eq:fin_Cap2} as is, one can replace the right-hand side of \eqref{eq:fin_Cap2} using \eqref{eq:fin_Cap2_RHS_qBin_alternative} and then apply Theorem~\ref{thm:Bailey} with $a=1$ in an iterative fashion. This yields an analogue of Theorem~\ref{thm:hierarchy_Fin_Cap2}.

\begin{theorem}\label{thm:hierarchy_Fin_Cap2_analogue} Let $L\in\mathbb{Z}_{\geq0}$, $f\in\mathbb{N}$ and $N_i := n_i+n_{i+1}+\dots+n_f$ with $i=1,2\dots, f$, then
\begin{align}
\nonumber
\sum_{m,n,n_1,n_2,\dots,n_f\geq 0 } &\frac{q^{2m^2+6mn+6n^2+m+3n +N_1^2 + N_2^2+\dots+N_f^2+N_1+N_2+\dots+N_f} (q;q)_{2L+1} (q;q)_{n_f}}{(q;q)_m (q^3;q^3)_n (q;q)_{L-N_1}(q;q)_{n_f-3n-2m }(q;q)_{n_1}(q;q)_{n_2}\dots(q;q)_{n_{f-1}}(q;q)_{2n_f+1}}\\ \label{eq:hierarchy_Fin_Cap2_analogue}
&\hspace{-.1cm}+q\sum_{m,n,n_1,n_2,\dots,n_f\geq 0 } \frac{q^{2m^2+6mn+6n^2+3m+6n +N_1^2 + N_2^2+\dots+N_f^2+N_1+N_2+\dots+N_f} (q;q)_{2L+1} (q;q)_{n_f}}{(q;q)_m (q^3;q^3)_n (q;q)_{L-N_1}(q;q)_{n_f-3n-2m-1}(q;q)_{n_1}(q;q)_{n_2}\dots(q;q)_{n_{f-1}}(q;q)_{2n_f+1}}\\ \nonumber
&\hspace{-1cm}  =  \sum_{j=-L}^L \Jacobi{j+1}{3} q^{(f+1)(j^2+j)} {2L+1\brack L-j}_q.
\end{align}
\end{theorem}

Letting $L\rightarrow\infty$ in \eqref{eq:hierarchy_Fin_Cap2_analogue} yields the analogue of Theorem~\ref{thm:hierarchy_Cap2} with the help of the quintuple product identity \eqref{eq:quintuple}.

\begin{theorem}\label{thm:hierarchy_Cap2_analogue} Let $L\in\mathbb{Z}_{\geq0}$, $f\in\mathbb{N}$ and $N_i := n_i+n_{i+1}+\dots+n_f$ with $i=1,2\dots, f$, then
\begin{align}
\nonumber
\sum_{m,n,n_1,n_2,\dots,n_f\geq 0 } &\frac{q^{2m^2+6mn+6n^2+m+3n +N_1^2 + N_2^2+\dots+N_f^2+N_1+N_2+\dots+N_f}  (q;q)_{n_f}}{(q;q)_m (q^3;q^3)_n (q;q)_{n_f-3n-2m }(q;q)_{n_1}(q;q)_{n_2}\dots(q;q)_{n_{f-1}}(q;q)_{2n_f+1}}\\ \nonumber%\label{eq:hierarchy_Cap2_analogue}
&+q\sum_{m,n,n_1,n_2,\dots,n_f\geq 0 } \frac{q^{2m^2+6mn+6n^2+3m+6n +N_1^2 + N_2^2+\dots+N_f^2 +N_1+N_2+\dots+N_f}  (q;q)_{n_f}}{(q;q)_m (q^3;q^3)_n (q;q)_{n_f-3n-2m-1}(q;q)_{n_1}(q;q)_{n_2}\dots(q;q)_{n_{f-1}}(q;q)_{2n_f+1}}\\ \nonumber
&\hspace{-1.5cm}  =   \frac{(q^{2(f+1)};q^{2(f+1)})_\infty}{(q;q)_\infty} 
(q^{2(f+1)},q^{10(f+1)};q^{12(f+1)})_\infty.
\end{align}
\end{theorem}
%
%Similarly, we can apply Theorem~\ref{thm:Bailey} to \eqref{eq:fin_Cap1_qBin_alternative} with $a=0$ and this yields the following companion to Theorem~\ref{thm:hierarchy_Fin_Cap1}.
%
%\begin{theorem}\label{thm:hierarchy_Fin_Cap1_alternate} Let $L\in\mathbb{Z}_{\geq0}$, $f\in\mathbb{N}$ and $N_i := n_i+n_{i+1}+\dots+n_f$ with $i=1,2\dots, f$, then
%\begin{align}
%\label{eq:hierarchy_Fin_Cap1_alternate}
%\sum_{m,n,n_1,n_2,\dots,n_f\geq 0 } &\frac{q^{2m^2+6mn+6n^2 +N_1^2 + N_2^2+\dots+N_f^2+N_f} (q;q)_{2L} (q;q)_{n_f}}{(q;q)_m (q^3;q^3)_n (q;q)_{L-N_1}(q;q)_{n_f-3n-2m }(q;q)_{n_1}(q;q)_{n_2}\dots(q;q)_{n_{f-1}}(q;q)_{2n_f}}\\ \nonumber
%&\hspace{3cm}  =  \sum_{j=-L}^L \Jacobi{j+1}{3} q^{(f+1)j^2-j} {2L\brack L-j}_q.
%\end{align}
%\end{theorem}
%
%Using the quintuple product identity \eqref{eq:quintuple}, Theorem~\ref{thm:hierarchy_Fin_Cap1_alternate} gives the follosing theorem as $L\rightarrow\infty$.
%
%\begin{theorem} Let $f\in\mathbb{N}$ and $N_i := n_i+n_{i+1}+\dots+n_f$ with $i=1,2\dots, f$, then
%\begin{align}
%\label{eq:hierarchy_Cap1_alternate}
%\sum_{m,n,n_1,n_2,\dots,n_f\geq 0 } &\frac{q^{2m^2+6mn+6n^2 +N_1^2 + N_2^2+\dots+N_f^2+N_f} (q;q)_{n_f}}{(q;q)_m (q^3;q^3)_n (q;q)_{n_f-3n-2m }(q;q)_{n_1}(q;q)_{n_2}\dots(q;q)_{n_{f-1}}(q;q)_{2n_f}}\\ \nonumber
%&\hspace{3cm} = \frac{(q^f,q^{5f+6},q^{6(f+1)};q^{6(f+1)})_\infty(q^{4f+6},q^{8f+6};q^{12(f+1)})_\infty}{(q;q)_\infty}.
%\end{align}
%\end{theorem}
%
%Finally, we would like to present a family of infinite hierarchies which also container Theorem~\ref{thm:hierarchy_Fin_Cap1_alternate} as its initial case (for $s=1$).

\begin{theorem}\label{thm:double_fin_hierarchy}  Let $f\in\mathbb{N}$ and $N_i := n_i+n_{i+1}+\dots+n_f$, with $i,s=1,2\dots, f$, where $N_{f+1}:=0$, then
\begin{align}
\label{eq:double_fin_hierarchy}
\sum_{m,n,n_1,n_2,\dots,n_f\geq 0 } &\frac{q^{2m^2+6mn+6n^2 +N_1^2 + N_2^2+\dots+N_f^2+N_{f-s+1} +\dots + N_f } (q;q)_{2L} (q;q)_{n_f}}{(q;q)_m (q^3;q^3)_n (q;q)_{L-N_1}(q;q)_{n_f-3n-2m }(q;q)_{n_1}(q;q)_{n_2}\dots(q;q)_{n_{f-1}}(q;q)_{2n_f}}\\ \nonumber
&\hspace{3cm}  =  \sum_{j=-L}^L \Jacobi{j+1}{3} q^{(f+1)j^2-sj} {2L\brack L-j}_q.
\end{align}
\end{theorem}

\begin{proof} We first apply Theorem~\ref{thm:Bailey} with $a=0$ to \eqref{eq:fin_Cap1_qBin_alternative}. This yields \begin{equation}
\label{eq:first_Bailey_application} 
\sum_{m,n,n_1\geq 0} \frac{q^{2m^2+  6 m n + 6n^2 + n_1^2+n_1} (q;q)_{n_1} (q;q)_{2L}}{(q;q)_m(q^3;q^3)_n(q;q)_{n_1-3n-2m}(q;q)_{L-n_1}(q;q)_{2n_1}} = \sum_{j=-L}^L \Jacobi{j+1}{3} q^{2j^2-j} {2L\brack L-j}_q.
\end{equation} Next, we multiply both sides of \eqref{eq:first_Bailey_application} with $q^L$ and apply \eqref{eq:k_transform} to the right-hand side of \eqref{eq:first_Bailey_application} with $k=2$ to get 
\begin{equation}
\label{eq:after_k_transform} 
\sum_{m,n,n_1\geq 0} \frac{q^{L+2m^2+  6 m n + 6n^2 + n_1^2+n_1} (q;q)_{n_1} (q;q)_{2L}}{(q;q)_m(q^3;q^3)_n(q;q)_{n_1-3n-2m}(q;q)_{L-n_1}(q;q)_{2n_1}} = \sum_{j=-L}^L \Jacobi{j+1}{3} q^{2j^2-2j} {2L\brack L-j}_q.
\end{equation} The outcome equation \eqref{eq:after_k_transform} is in a form suitable for the application of Theorem~\ref{thm:Bailey} with $a=0$. We can once again apply Bailey's lemma with $a=0$, directly follow this step with multiplying both sides with $q^L$, and rewriting the right-hand side of the outcome expression with \eqref{eq:k_transform} with $k=3$ as in the previous case. This yields the next step in this succession:
\begin{equation*}
\sum_{m,n,n_1,n_2\geq 0} \frac{q^{L+2m^2+  6 m n + 6n^2 + (n_1+n_2)^2+n_2^2+(n_1+n_2)+n_2} (q;q)_{n_2} (q;q)_{2L}}{(q;q)_m(q^3;q^3)_n(q;q)_{n_2-3n-2m}(q;q)_{L-(n_1+n_2)}(q;q)_{n_1}(q;q)_{2n_2}} = \sum_{j=-L}^L \Jacobi{j+1}{3} q^{3j^2-3j} {2L\brack L-j}_q.
\end{equation*} 
Proceeding in this fashion, for any $s=1,2,\dots$ we arrive at
\begin{align}
\nonumber\sum_{m,n,n_1,\dots,n_s\geq 0} &\frac{q^{2m^2+  6 m n + 6n^2 + N_1^2+N_2^2+\dots+N_s^2+N_1+N_2+\dots+N_s} (q;q)_{n_s} (q;q)_{2L}}{(q;q)_m(q^3;q^3)_n(q;q)_{n_s-3n-2m}(q;q)_{L-N_1}(q;q)_{n_1}\dots(q;q)_{n_{s-1}}(q;q)_{2n_s}}\\\label{eq:s_Bailey_applications} &\hspace{7cm}= \sum_{j=-L}^L \Jacobi{j+1}{3} q^{(s+1)j^2-sj} {2L\brack L-j}_q,
\end{align} where $N_i:=n_i+ ….+n_s$ for $i=1,2,..,s$.
Finally, applying Theorem~\ref{thm:Bailey} with $a=0$ directly to \eqref{eq:s_Bailey_applications}, without going through the multiplication with the factor $q^L$,  $f-s$ times (for any $f\geq s$) yields~\eqref{eq:double_fin_hierarchy}.
\end{proof}

We would like to remark that Theorem~\ref{thm:hierarchy_Fin_Cap1} can be seen as the $s=0$ case of Theorem~\ref{thm:double_fin_hierarchy}. Also note that if one apply Theorem~\ref{thm:Bailey} with $a=0$ iteratively to \eqref{eq:fin_Cap1_qBin_alternative}, then one gets the $s=1$ case Theorem~\ref{thm:double_fin_hierarchy}.

Tending $L\rightarrow\infty$ in \eqref{eq:double_fin_hierarchy} and using \eqref{eq:quintuple} we get Theorem~\ref{thm:hierarchy_Cap1_alternate} for $s\in \mathbb{N}$. The $s=0$ case of Theorem~\ref{thm:hierarchy_Cap1_alternate} is Theorem~\ref{thm:hierarchy_Cap1} and it is already proven.

\section{Dual Identities}\label{Sec:dual}

Replacing $q\mapsto 1/q$ in \eqref{eq:fin_Cap1} and \eqref{eq:fin_Cap2}, using \eqref{eq:qBinom_1_over_q} followed by multiplying both sides with $q^{L^2}$ and $q^{L^2+L}$, respectively, yields the following theorem.

\begin{theorem}\label{thm:dual_fin_Cap1}
\begin{align}
\label{eq:dua_fin_Cap1} \sum_{\substack{n,m\geq0 \\ L\equiv n+2m\text{ (mod }3)}} \frac{(-1)^m q^{\frac{m(m-1)}{2}+Ln} (q;q)_L}{(q;q)_m (q;q)_n (q^3;q^3)_{(L-n-2m)/3}} &= \sum_{j = -\infty}^{\infty} \Jacobi{j+1}{3} {2L\brack L-j}_q,\\
\label{eq:dua_fin_Cap2} \sum_{\substack{n,m\geq0 \\ L\equiv n+2m\text{ (mod }3)}} \frac{(-1)^m q^{\frac{m(m+1)}{2}+(L+1)n} (q;q)_L}{(q;q)_m (q;q)_n (q^3;q^3)_{(L-n-2m)/3}} &-\sum_{\substack{n,m\geq0 \\ L\equiv n+2m+1\text{ (mod }3)}} \frac{(-1)^m q^{\frac{m(m+1)}{2}+(L+1)n} (q;q)_L}{(q;q)_m (q;q)_n (q^3;q^3)_{(L-n-2m-1)/3}}\\ \nonumber
&\hspace{1cm}= \sum_{j = -\infty}^{\infty} \Jacobi{j+1}{3} q^{L-j} {2L\brack L-j}_q,
\end{align} where $\Jacobi{\cdot}{\cdot}$ is the Jacobi symbol.
\end{theorem}

The asymptotic behaviors of \eqref{eq:dua_fin_Cap1} and \eqref{eq:dua_fin_Cap2} as $L\rightarrow\infty$ should also be considered. %It is easy to see that right-hand side of \eqref{eq:dua_fin_Cap1} diverges when we tend $L$ to infinity. Therefore, we will only focus on the limits related to \eqref{eq:dua_fin_Cap2} from now on.
It is easier to understand the asymptotic behavior of these identities through \eqref{eq:dua_fin_Cap2}, so we will only be focusing on that.

All the summands of the left-hand side series of \eqref{eq:dua_fin_Cap2} with any non-zero $n$-values vanishes as $L\rightarrow \infty$ for $|q|<1$. This reduces the double sums on the left-hand sides of \eqref{eq:dua_fin_Cap2} to single sums for the limit discussions. 

Letting $L\mapsto 3L$, we see that the left-hand side summation conditions '$L\equiv n+2m\text{ (mod }3)$' and '$L\equiv n+2m+1\text{ (mod }3)$' imply that '$m\equiv 0$ (mod 3)' and '$m\equiv 1$ (mod 3)', respectively. On the right-hand side of \eqref{eq:dua_fin_Cap2}, we do the simple $k=3L-j$ substitution before taking any limits. We then see that $L\rightarrow\infty$ implies

\begin{equation}\label{eq:limit_3L_of_fin_cap2_raw}
\frac{(q;q)_\infty}{(q^3;q^3)_\infty} \left( \sum_{\substack{m\geq 0 \\ m\equiv0\text{ (mod }3)} }(-1)^m \frac{q^{\frac{m(m+1)}{2}}}{(q;q)_m} - \sum_{\substack{m\geq 0 \\ m\equiv1\text{ (mod }3)}} (-1)^m \frac{q^{\frac{m(m+1)}{2}}}{(q;q)_m} \right) = \sum_{k\geq 0} \Jacobi{1-k}{3} \frac{q^k}{(q;q)_k}.
\end{equation}

The identity \eqref{eq:limit_3L_of_fin_cap2_raw} can be simplified by writing the left-hand side sums explicitly and then combining the terms. For the right-hnd side we can use the simple observation about the Jacobi symbols that $\Jacobi{1-k}{3} = -\Jacobi{k+2}{3}$. Then we get,

\begin{corollary}\label{cor:limit_3L_fin_Cap2}
\begin{equation}\label{eq:limit_3L_fin_Cap2}
\frac{(q;q)_\infty}{(q^3;q^3)_\infty} \sum_{m\geq 0 } (-1)^{m+1} \frac{q^{\frac{3m(3m+1)}{2}}}{(q;q)_{3m+1}} = \sum_{k\geq 0 } \Jacobi{k+2}{3} \frac{q^k}{(q;q)_k}.
\end{equation}
\end{corollary}

Similar to Corollary~\ref{cor:limit_3L_fin_Cap2}, we can take the limit of \eqref{eq:dua_fin_Cap2} after $L\mapsto 3L+2$ and $L\mapsto 3L+1$. These considerations imply the two identities of the following corollary, respectively.

\begin{corollary}\label{cor:limit_3L+2and+1_fin_Cap2}
\begin{align}
\label{eq:limit_3L+2_fin_Cap2}\frac{(q;q)_\infty}{(q^3;q^3)_\infty} \sum_{m\geq 0 } (-1)^m \frac{q^{\frac{(3m+1)(3m+2)}{2}}}{(q;q)_{3m+2}} &= \sum_{k\geq 0 } \Jacobi{k}{3} \frac{q^k}{(q;q)_k},\\
\label{eq:limit_3L+1_fin_Cap2}\frac{(q;q)_\infty}{(q^3;q^3)_\infty} \sum_{m\geq 0 } (-1)^m \frac{q^{\frac{(3m)(3m-1)}{2}}}{(q;q)_{3m}} &= \sum_{k\geq 0 } \Jacobi{k+1}{3} \frac{q^k}{(q;q)_k}.
\end{align}
\end{corollary}

We can combine and rewrite the equations \eqref{eq:limit_3L_fin_Cap2}-\eqref{eq:limit_3L+1_fin_Cap2} as
\begin{equation*}
\sum_{k\geq 0} \Jacobi{k+b}{3} \frac{q^k}{(q;q)_k} = \frac{(q;q)_\infty}{(q^3;q^3)_\infty} \sum_{m\geq 0 } (-1)^{m+1} \Jacobi{m-b}{3} \frac{q^{\frac{m(m+1)}{2}}}{(q;q)_{m}},
\end{equation*} where $b=0,1,2$. This identity can also be proven directly by appeal to the $q$-binomial theorem \eqref{eq:qBin_thm}.

Corollaries~\ref{cor:limit_3L_fin_Cap2} and \ref{cor:limit_3L+2and+1_fin_Cap2} can be interpreted as weighted partition theorems. For example, \eqref{eq:limit_3L+2_fin_Cap2} has the following partition theoretic interpretation.

\begin{theorem}\label{thm:Weighted_1} Let
\begin{align*}
P_1 &:= \{ \pi \in D\, :\, \#(\pi)\not\equiv 0\text{ (mod }3)\},\\
P_2 &:= \{ \pi=(\lambda_1,\dots) \in P\, :\,  3\not|\, \lambda_1\lambda_2\dots\lambda_{\#(\pi)}\},\\
P_3 &:=  \{ \pi \in P\, :\, \#(\pi)\not\equiv 0\text{ (mod }3)\},\\
\end{align*}
then
\begin{equation}\label{eq:abstract_weighted_1}
\sum_{\pi\in P_1} (-1)^{\mu(\pi)}q^{|\pi|} = \sum_{(\pi_1,\pi_2)\in P_2\times P_3} (-1)^{\sigma(\pi_2)} q^{|\pi_1|+|\pi_2|},
\end{equation}
where $\sigma(\pi)$ is 1 if $\#(\pi)\equiv 2$ (mod 3) and 0 otherwise, and $\mu(\pi) := \#(\pi)+\sigma(\pi)+1$.
\end{theorem}

For example, for partitions of 3, the left-hand side of \eqref{eq:abstract_weighted_1} counts the parititons $(3)$ and $(2,1)$. Both of these partitions are counted with positive weight and grants the total count of 2. On the right-hand side we consider the following pair of partitions with a total size of 3 and their respective weights $\omega$: 
\[\begin{array}{cc|cc|cc}
(\pi_1,\pi_2)\in P_2\times P_3 & \omega & (\pi_1,\pi_2)\in P_2\times P_3 & \omega& (\pi_1,\pi_2)\in P_2\times P_3 & \omega\\\hline
\left( (2), (1) \right) & 1 & \left( (1,1), (1) \right) & 1 & \left( (1), (2)\right) & 1 \\
\left( (1), (1,1) \right) & -1 & \left( \emptyset, (3) \right) & 1 & \left( \emptyset, (2,1)\right) & -1 \\
\end{array}
\]
The total of all these weights is also 2.

To prove Theorem~\ref{thm:Weighted_1}, we rewrite \eqref{eq:limit_3L+2_fin_Cap2} as \[\sum_{\substack{m\geq 0\\ m\equiv 2\mod{3} }} (-1)^m \frac{q^{m(m+1)/2}}{(q;q)_m} -\sum_{\substack{m\geq 0\\ m\equiv 1\mod{3} }} (-1)^m \frac{q^{m(m+1)/2}}{(q;q)_m} = \frac{1}{(q,q^2;q^3)_\infty} \sum_{k\geq 0} \Jacobi{k}{3} \frac{q^k}{(q;q)_k}.
\]

In light of \eqref{eq:GF_P_D}, the left-hand side sums are the generating functions for distinct partitions with number of parts $\equiv 2$ modulo 3 and $\equiv 1$ modulo 3 parts, respectively, where $-1$ is raised to the number of parts with one extra negative sign for partitions with number of parts congruent to 1 modulo 3. On the right-hand side, it is clear that the $(q,q^2;q^3)_\infty^{-1}$ is the generating function for partitions where no part is divisible by 3. The last sum on the far right is the generating function for partitions with number of parts $\not\equiv 0$ modulo 3 where partitions with 2 modulo 3 number of parts are weighted with a $-1$ coming from the Jacobi symbol.

Similarly, \eqref{eq:limit_3L_fin_Cap2} and \eqref{eq:limit_3L+1_fin_Cap2} can be interpreted as the two following weighted partition theorems, respectively.
\begin{theorem}\label{thm:Weighted_2} Let
\[
P^*_1 := \{ \pi \in D\, :\, \#(\pi)\not\equiv 2\text{ (mod }3)\},\text{  and  }
P^*_3 :=  \{ \pi \in P\, :\, \#(\pi)\not\equiv 1\text{ (mod }3)\},
\]
then
\begin{equation*}%\label{eq:abstract_weighted_2}
\sum_{\pi\in P^*_1} (-1)^{\mu^*(\pi)}q^{|\pi|} = \sum_{(\pi_1,\pi_2)\in P_2\times P^*_3} (-1)^{\sigma^*(\pi_2)} q^{|\pi_1|+|\pi_2|},
\end{equation*}
where $\sigma^*(\pi)$ is 1 if $\#(\pi)\equiv 0$ (mod 3) and 0 otherwise, and $\mu^*(\pi) := \#(\pi)+\sigma^*(\pi)$.
\end{theorem}

\begin{theorem}\label{thm:Weighted_3} Let
\[
P'_1 := \{ \pi \in D\, :\, \#(\pi)\not\equiv 1\text{ (mod }3)\}\text{  and  }
P'_3 :=  \{ \pi \in P\, :\, \#(\pi)\not\equiv 2\text{ (mod }3)\},\\
\]
then
\begin{equation*}%\label{eq:abstract_weighted_3}
\sum_{\pi\in P'_1} (-1)^{\mu^*(\pi)}q^{|\pi|} = \sum_{(\pi_1,\pi_2)\in P_2\times P'_3} (-1)^{\sigma^*(\pi_2)} q^{|\pi_1|+|\pi_2|},
\end{equation*}
where $\sigma^*(\pi)$ is 1 if $\#(\pi)\equiv 0$ (mod 3) and 0 otherwise, and $\mu^*(\pi) := \#(\pi)+\sigma^*(\pi)$.
\end{theorem}

\section{Outlook}

Although the study here is systematic and seemingly complete, there are still some interesting leads to be explored and missing cases to be found. For example, we could not find a result analogous to \eqref{eq:Fin_Sum_of_capparelli's} for Theorem~\ref{thm:New_Fin_Cap}. Also the simple extension of \eqref{eq:Fin_Cap1_Binomial} analogous to \eqref{eq:hierarchy_Cap1_alternate} does not seem to exist.

However, the missing applications of Bailey's lemma to Theorem~\ref{thm:dual_fin_Cap1} is not an oversight. Instead it is a deliberate choice that we make. Applying Theorem~\ref{thm:Bailey} to Theorem~\ref{thm:dual_fin_Cap1} yields the same polynomial right-hand sides the application of Bailey's lemma to Theorem~\ref{thm:New_Fin_Cap} with messier polynomials on the left-hand side.

\section*{Acknowledgements}

The authors would like to thank Krishnaswami Alladi, George E. Andrews, Peter Paule, and Wadim Zudilin for their genuine interest, encouragement, and helpful comments.


\begin{thebibliography}{99}

\bibitem{qFunctions} J. Ablinger and A. K. Uncu, \texttt{qFunctions }\textit{- A Mathematica package for $q$-series and partition theory applications}, Journal of Symbolic Computation 107, (2021), 145-166.

\bibitem{Refinement} K. Alladi, G. E. Andrews, and B. Gordon, \textit{Refinements and Generalizations of Capparelli's Conjecture on Partitions}, Journal of Algebra \textbf{174} (1995), no. 2, 636--658.

\bibitem{GEA_multi} G. E. Andrews, \textit{Multiple series Rogers--Ramanujan type identities}, Pacific. J. Math.114, 267-283 (1984).

\bibitem{Andrews_Baxter} G. E. Andrews, and R. J. Baxter, \textit{Lattice gas generalization of the hard hexagon model. III. q-Trinomial coefficients}, J. Statist. Phys. 47 (1987), no: 3-4, 297-330.

\bibitem{Andrews_Capparelli}G. E. Andrews, \textit{Schur's theorem. Capparelli's conjecture and q-trinomial coefficients}, Contemporary Mathematics \textbf{166} (1994), 141--154.

\bibitem{Theory_of_Partitions}G. E. Andrews, \textit{The theory of partitions}, Cambridge Mathematical Library, Cambridge University Press, Cambridge, 1998. Reprint of the 1976 original. MR1634067 (99c:11126)

\bibitem{BU_Finite_Capparelli} A. Berkovich and A. K. Uncu, \textit{Polynomial identities implying Capparelli's partition theorems},  J. Number Theory 201 (2019), 77-107. 

\bibitem{BU_Elementary} A. Berkovich and A. K. Uncu, \textit{Elementary polynomial identities involving $q$-trinomial coefficients}, Ann. Comb. 23 (2019), no. 3-4, 549-560.

\bibitem{BU_infinite_Fam} A. Berkovich and A. K. Uncu, \textit{Refined $q$-Trinomial Coefficients and Two Infinite Hierarchies of $q$-Series Identities},  Algorithmic Combinatorics: Enumerative Combinatorics, Special Functions and Computer Algebra. Texts \& Monographs in Symbolic Computation (A Series of the Research Institute for Symbolic Computation, Johannes Kepler University, Linz, Austria). Springer, Cham. https://doi.org/10.1007/978-3-030-44559-1\_4

\bibitem{Capparelli} S. Capparelli, \textit{A combinatorial proof of a partition identity related to the level 3 representation of twisted affine Lie algebra}, Communications in Algebra \textbf{23} (1995), no. 8, 2959-2969.

\bibitem{Capparelli_thesis} S. Capparelli, \textit{Vertex operator relations for affine algebras and combinatorial identities}, Ph.D Thesis Rutgers University (1988).

\bibitem{Dousse1} J. Dousse, \textit{On partition identities of Capparelli and Primc}, Adv. Math. 370 (2020), 107245. 

\bibitem{Dousse2} J. Dousse, and J. Lovejoy, \textit{Generalizations of Capparelli's identity}, Bull. Lond. Math. Soc. 51, Issue 2 (2019), pp. 193-206. 
     
\bibitem{Gasper_Rahman} G. Gasper and M. Rahman, \textit{Basic hypergeometric series}, Cambridge University Press, 2004.


\bibitem{Kanade_Russell} S. Kanade and M. Russell, \textit{Staircases to analytic sum-sides for many new integer partition identities of
Rogers--Ramanujan type}, Electron. J. Combin. 26 (2019), no. 1, Paper 1.6.


\bibitem{qGeneratingFunctions} M. Kauers and C. Koutschan, \textit{A Mathematica package for q-holonomic sequences and power series}, The Ramanujan Journal, 19 (2), 137-150, Springer, 2009, ISSN 1382-4090.


\bibitem{HolonomicFunctions} C. Koutschan, \textit{Advanced Applications of the Holonomic Systems Approach}, RISC, Johannes Kepler University, Linz. PhD Thesis. September 2009.

\bibitem{Kagan} K. Kur\c{s}ung\"oz, \textit{Andrews--Gordon type series for Capparelli's and G\"ollnitz--Gordon identities}, J.Combin.Theory Ser.A 165 (2019), 117-138.

\bibitem{AeqB} M. Petkovšek, H. S. Wilf, and D. Zeilberger, \textit{$A=B$} (With a foreword by Donald E. Knuth. With a separately available computer disk). A K Peters, Ltd., Wellesley, MA, 1996. xii+212 pp. ISBN: 1-56881-063-6

\bibitem{qZeil} P. Paule and A. Riese, \textit{A Mathematica q-Analogue of Zeilberger’s Algorithm Based on an Algebraically Motivated Approach to q-Hypergeometric Telescoping, in Special Functions, q-Series and Related Topics}, Fields Inst. Commun., Vol. 14, pp. 179-210, 1997.

\bibitem{Peter_Thesis} P. Paule, \textit{Zwei neue Transformationen als elementare Anwendungen derq-Vandermonde Formel}, Ph.D. Thesis (1982),University of Vienna.

\bibitem{qMultiSum} A. Riese, \textit{qMultiSum - A Package for Proving q-Hypergeometric Multiple Summation Identities}, Journal of Symbolic Computation 35 (2003), 349-376.

\bibitem{Sigma} C. Schneider, \textit{Symbolic Summation Assists Combinatorics}, Sem.Lothar.Combin. \textbf{56}, (2007), pp.1-36. Article B56b.

%\bibitem{Sills}A. V. Sills, \textit{On series expansions of Capparelli's infinite product}, Advances in Applied Mathematics \textbf{33} (2004), no. 2 397-408.

\bibitem{Tamba} M. Tamba, C. F. Xie \textit{Level three standard modules for $A_2^2$ and combinatorial identities}, J. Pure Appl. Algebra 105 (1995), no. 1, 53–92.

\bibitem{Warnaar_T} S. O. Warnaar, \textit{The generalized Borwein conjecture. II. Refined $q$-trinomial coefficients}, Discrete Math. \textbf{272} (2003), no. 2-3, 215-258.

   \end{thebibliography}
\end{document}